\documentclass{amsart}%
\usepackage{amsmath}
\usepackage{amssymb}
\usepackage{amsfonts}
\usepackage{graphicx}%
\providecommand{\U}[1]{\protect\rule{.1in}{.1in}}
\newtheorem{theorem}{Theorem}

\newtheorem{lemma}[theorem]{Lemma}

\newtheorem{proposition}[theorem]{Proposition}

\begin{document}
\title{Orbital Measures on $SU(2)/ SO(2)$}
\author{Boudjem\^aa Anchouche}
\author{Sanjiv Kumar Gupta}
\author{Alain Plagne}
\address{B. Anchouche, Department of Mathematics and Statistics, Sultan Qaboos
University, PO Box 36, Al-Khoud 123, Muscat, Sultanate of Oman}
\email{anchouch@squ.edu.om}
\address{S. K. Gupta, Department of Mathematics and Statistics, Sultan Qaboos
University, PO Box 36, Al-Khoud 123, Muscat, Sultanate of Oman}
\email{gupta@squ.edu.om}
\address{A. Plagne, Centre de Math\' ematiques Laurent Schwartz, \' Ecole
polytechnique, 91128 Palaiseau Cedex, France}
\email{plagne@math.polytechnique.fr}
\thanks{The first and second authors are supported by the Sultan Qaboos University Grants}
\thanks{The third author is supported by the ANR Caesar grant number ANR-12-BS01-0011}
\date{\today }

\begin{abstract}
We let $U=SU(2)$ and $K=SO(2)$ and denote $N_{U}\left(  K\right)  $ the
normalizer of $K$ in $U$. For $a$ an element of $U\backslash N_{U}\left(
K\right)  $, we let $\mu_{a}$ be the normalized singular measure supported in
$KaK$. For $p$ a positive integer, it was proved in \cite{gupta-hare1} that
$\mu_{a}^{\left(  p\right)  }$, the convolution of $p$ copies of $\mu_{a},$ is
absolutely continuous with respect to the Haar measure of the group $U$ as
soon as $p\geq2$. The aim of this paper is to go a step further by proving the
following two results :
(i) for every $a$ in $U\backslash N_{U}\left(  K\right)  $ and every integer
$p\geq3$, the Radon-Nikodym derivative of $\mu_{a}^{\left(  p\right)  }$ with
respect to the Haar measure $m_{U}$\ on $U$, namely $\mathrm{d}\mu
_{a}^{\left(  p\right)  }/\mathrm{d}m_{U}$, is in $L^{2}(U)$, and
(ii) there exist $a$ in $U\backslash N_{U}\left(  K\right)  $ for which
$\mathrm{d}\mu_{a}^{\left(  2\right)  }/\mathrm{d}m_{U}$ is not in $L^{2}(U),$
hence a counter example to the dichotomy conjecture stated in
\cite{gupta-hare3}. Since $L^{2}\left(  G\right)  \subseteq L^{1}\left(  G\right)  $, our result
gives in particular a new proof of the result in \cite{gupta-hare1} when
$p>2.$

\end{abstract}
\maketitle

\section{Introduction}

Let $M$ be a symmetric space of non compact type and $G$ be the identity component
of the isometry group of $M.$ Let $p$ be in $M$ and let $K$ be the subgroup of
$G$ fixing $p$ that is, $K=\left\{  g\in G\mid gp=p\right\}$. It is well known
that $G$ is a semisimple Lie group with trivial center and no compact factor
and $K$ is a maximal compact subgroup of $G$ (see \cite{eb} or \cite{hel}). Moreover, the
map $\zeta:G/K\rightarrow M$, defined by $\zeta\left(  gK\right)  =gp$ is a
diffeomorphism, and if we endow $G/K$ with the pull back of the metric of $M$,
then $\zeta$ becomes an isometry.

Let $\mathfrak{g}=\mathfrak{k}\oplus\mathfrak{p}$ be a Cartan decomposition of
the Lie algebra $\mathfrak{g}$ of $G.$ Let now $\mathfrak{u=k}\oplus i$
$\mathfrak{p}$ and let $U$ be a Lie group with Lie algebra $\mathfrak{u}$.
Then $U$ is a compact group and $\widetilde{M}=U/K$ is a compact symmetric
space, called the dual of $M.$

Let $\mathfrak{a}$ be a maximal abelian subspace of $\mathfrak{p,}$ $H$ be an
element of $\mathfrak{a}$ and $a=\exp\left(  iH\right)$. We finally define $\mu_{a}$ 
to be the normalized singular measure supported in the double coset $KaK$. For a
positive integer $p$, we denote by $\mu_{a}^{\left(  p\right)  }$ the convolution
of $p$ copies of $\mu_{a}$ and we let
\[
p_{U/K}\left(  a\right)  =\min\left\{  p\in \mathbb{N} \text{ such that }\mu_{a}^{\left(  p\right)  }\text{ is absolutely continuous with respect to } m_{U}\right\}
\]
where $m_{U}$ is the Haar measure of the group $U$. 

In the case $U=SU\left(  n\right)$ and $K=SO\left(  n\right)$, 
the second author and K. Hare proved in \cite{gupta-hare1} that for any point $a$ 
not in the normalizer of $SO\left(  n\right)$ in $SU\left(  n\right)$, $p_{SU(n)/SO(n)}\left(  a\right)  $ is equal to
the rank of the symmetric space $SU(n)/SO(n)$, that is, 
$$
p_{SU(n)/SO(n)}\left(a\right)  =n.
$$

In \cite{gupta-hare3}, S. K. Gupta and K. Hare conjectured that the Radon-Nikodym derivative
of $\mu_{a}^{\left(  p\right) }$ with respect to the Haar measure $m_{U}$, 
which is thereafter denoted by $\mathrm{d}\mu_{a}^{\left(  p\right)  }/\mathrm{d}m_{U}$, 
is in $L^{2}\left( U\right)$ as soon as $p\geq p_{U/K}\left(  a\right)$. 

In this paper, we start the investigation of the $L^{2}-$regularity of
$\mathrm{d}\mu_{a}^{\left(  p\right)  }/\mathrm{d}m_{SU\left(  2\right)}$
since it turns out that already the case of $SU(2)$ needs quite a
lot of efforts.

As a first result, we give a counter example to the Gupta-Hare conjecture mentioned above
in the case $U=SU(2)$. 
We also prove that
$$
\frac{\mathrm{d}\mu_{a}^{\left(  p\right)}}{\mathrm{d}m_{SU\left(  2\right)}} \in L^{2}\left( SU\left(  2\right) \right)
$$ 
as soon as $p\geq 3$ and, as a corollary, we give a new proof of the result quoted above 
on the absolute continuity of $\mu_{a}^{\left(  p\right)  }$ in the case $p\geq3$ and $n=2.$ 

Here is the plan of the present article. In  Section \ref{sec2}, we say a bit more about 
the normalized singular measure supported in the double coset $KaK$ and 
announce our main result (Theorem \ref{Main}). In section \ref{sec3}, we
present the general setting of Fourier transform of orbital measures that we
make precise in the case of $SU(2)/SO(2)$ in Section \ref{sec4}. Sections
\ref{sec5} and \ref{sec6} contain the two parts of the proof of our main theorem.

\section{Orbital Measures and the Main Theorem}
\label{sec2}

Let $K$ and $U$ as above and consider the natural action of $K\times K$ on $U
$ defined as follows%
\[%
\begin{array}
[c]{cccc}%
\sigma: & \left(  K\times K\right)  \times U & \longrightarrow & U\\
& \left(  \left(  k_{1},k_{2}\right)  ,a\right)  & \longmapsto & k_{1}ak_{2}.
\end{array}
\]
The orbit of a point $a\in U$ under the action $\sigma$ is the double coset
space $KaK.$ Since $U=KAK,$ where $A=\exp\left(  i\mathfrak{a}\right)  $ (see
\cite{AWK}, p. $458$), we will assume in all what follows that $a=\exp\left(
iH\right)  $, with $H$ in $\mathfrak{a}$.

Each orbit is equipped with a unique $K\times K$-invariant ($K-$bi-invariant)
measure $\mu_{a}$, defined as follows : for any $f$ in $C(U)$, the set of
continuous functions on $U$, we put
\[
\left\langle \mu_{a},f\right\rangle =\int\limits_{U}f\left(  x\right)
\mathrm{d}\mu_{a}\left(  x\right)  =\int\limits_{K}\int\limits_{K}
f(k_{1}ak_{2})\mathrm{d}m_{K}(k_{1})\mathrm{d}m_{K}(k_{2}).
\]

The support of the measure $\mu_{a},$ denoted by $\operatorname{supp}\left(  \mu_{a}\right)$, 
is $KaK$.
Since $\operatorname{supp}\left(  \mu_{a}\right)  =KaK$ has an empty\ interior
in $U$, the measure $\mu_{a}$ is singular with respect to the Haar measure
$m_{U}$ of the group $U$. Clearly, $\mu_{a}$ is continuous if and if only $a$
is not in the normalizer of $K$ in $U.$

We recall that the convolution of two measures $\mu$ and $\nu$ on $U,$ denoted
by $\mu\ast\nu,$ is defined, for any $f\in C\left(  U\right)  $, by
\[
\left\langle \mu\ast\nu,f\right\rangle =\int_{U}f\left(  x\right)
\mathrm{d}\left(  \mu\ast\nu\right)  =\int_{U}\int_{U}f(xy)\mathrm{d}
\mu(x)\mathrm{d}\nu(y).
\]
Denote the convolution of $p$ copies of a given measure $\mu$, namely $\mu
\ast\mu\ast...\ast\mu$ (with $p-1$ signs $\ast$) by $\mu^{(p)}$. Then for
$f\in C\left(  U\right)  ,$
\[
\left\langle \mu_{a}^{\left(  p\right)  },f\right\rangle =\int_{U}\cdots
\int_{U}f\left(  g_{1}\dots g_{p}\right)  \mathrm{d}\mu_{a}\left(
g_{1}\right)  \cdots\mathrm{d}\mu_{a}\left(  g_{p}\right)  .
\]
Since $\operatorname{supp}\left(  \mu_{a}\right)  =KaK,$ we have
\[
\operatorname{supp}\left(  \mu_{a}^{\left(  p\right)  }\right)  =\left(
KaK\right)  ^{p}.
\]

Specializing what we mentioned above, it was proved in \cite{gupta-hare1} that
$\mu_{a}^{\left(  p\right)  }$ is absolutely continuous with respect to the
Haar measure if and only if $p\geq2$. Then the Radon-Nikodym Theorem asserts
the existence of a function $f_{a,p}\in L^{1}\left(  U,\mathrm{d}m_{U}\right)
$ such that for $p\geq2,$
\[
\mathrm{d}\mu_{a}^{\left(  p\right)  }=f_{a,p}\mathrm{d}m_{U}.
\]

Here, we shall prove that $\mathrm{d}\mu_{a}^{\left(  p\right)  }/\mathrm{d}
m_{U}$ is actually in $L^{2}(U)$ provided $p\geq3$. More precisely, we prove
the following result, which is our main result.

\begin{theorem}
\label{Main} 
Let $U=SU(2)$, $K=SO(2)$ and denote $N_{U}\left(  K\right)  $ the
normalizer of $K$ in $U$. For $a$ an element of $U\backslash N_{U}\left(
K\right)  $, let $\mu_{a}$ denote the normalized singular measure supported in
$KaK$. Then

\begin{enumerate}
\item[(i)] for $p\geq3$, the measure $\mu_{a}^{(p)}$ is absolutely continuous
with respect to the Haar measure on $U$ and $\mathrm{d}\mu_{a}^{\left(
p\right)  }/\mathrm{d}m_{U}$ is in $L^{2}(U)$,

\item[(ii)] there exist an element $a\in U\backslash N_{U}\left(  K\right)  $
for which $\mathrm{d}\mu_{a}^{\left(  2\right)  }/\mathrm{d}m_{U}$ is not in
$L^{2}(U)$.
\end{enumerate}
\end{theorem}

Sections \ref{sec5} contains the proof of point (i) of Theorem \ref{Main}
while Section \ref{sec6} contains the one of (ii) in a stronger form (much more than 
one element $a$ will be exhibited for which $\mathrm{d}\mu_{a}^{\left(  2\right)  }/\mathrm{d}m_{U}$ is not in
$L^{2}(U)$).

This work is motivated by the article of the second-named author (joint with
K. Hare) in \cite{gupta-hare2} in which it was proved that if $G$ is a
compact, simple Lie group and $a$ is not in the center of $G$, then $\mu_{a},$
the measure supported on the conjugacy class of $a,$ satisfies the following
dichotomy: for each natural number $p,$
\[
\text{either }\mu_{a}^{\left(  p\right)  }\text{ is singular}\hspace
{0.5cm}\text{or}\hspace{0.5cm}\mu_{a}^{\left(  p\right)  }\in L^{2}(G)
\]
In other words,
\begin{equation}
\mu_{a}^{\left(  p\right)  }\in L^{1}(G)\hspace{0.5cm}\text{if and only
if}\hspace{0.5cm}\mu_{a}^{\left(  p\right)  }\in L^{2}(G). \label{star}%
\end{equation}
Since compact simple Lie groups can be seen as compact symmetric spaces by
identifying $G$ with $(G\times G)/\Delta$, where $\Delta$ is the diagonal of
$G\times G$, these authors further conjectured in \cite{gupta-hare3} that the
dichotomy above holds for any compact symmetric space$.$ More precisely, they
conjectured that if $U/K$ is a compact symmetric space and $a$ is not in the
normalizer of $K$ in $U$ and $\mu_{a}$ is the normalized singular measure
supported in $KaK,$ then \eqref{star} holds.

Theorem \ref{Main} shows that the dichotomy conjecture is false on the compact
symmetric space $SU(2)/SO(2)$.

\section{Fourier Transform of Orbital Measures}
\label{sec3}

In this section, $U/K$ will denote an arbitrary compact symmetric space. For
each irreducible representation $\pi:U\rightarrow GL\left(  E_{\pi}\right)  $
of $U$, we fix a $U$-invariant inner product\ in $E_{\pi},$ which we denote
for simplicity by $\left(  .,.\right)  $ and denote by $\left\Vert
.\right\Vert $ the corresponding norm. The set of equivalence classes of
irreducible unitary representations of $U$ will be denoted by $\widehat{U}$.
An irreducible unitary representation $\left(  \pi,E_{\pi}\right)  $ is called
\emph{of class one}, or \emph{spherical}, if $\dim E_{\pi}^{K}=1,$ where
\[
E_{\pi}^{K}=\left\{  X\text{ in }E_{\pi}\mid\pi\left(  k\right)  X=X\text{ for
all }k\text{ in }K\right\}  .
\]
The Fourier transform of $\mu_{a}$ at an irreducible unitary representation
$\pi:U\rightarrow GL(E_{\pi}),$ denoted by $\widehat{\mu_{a}}\left(
\pi\right)  ,$ is an element of $\mathrm{End}(E_{\pi})$ given by (see
\cite{hr}, p. $77$),
\[
\widehat{\mu_{a}}\left(  \pi\right)  \left(  X\right)  :=\int_{U}\pi\left(
g^{-1}\right)  X\mathrm{d}\mu_{a}\left(  g\right)  .
\]
Note first that the $U$-invariant inner product $\left(  .,.\right)  $\ in
$E_{\pi}$ induces a Hilbert structure on $\mathrm{End}(E_{\pi})$ defined as
follows: we define the Hilbert-Schmidt inner product of two elements $S$,
$T\in\mathrm{End}(E_{\pi}),$ denoted $\left(  S,T\right)  _{\mathrm{HS}},$ by
the formula
\[
\left(  S,T\right)  _{\mathrm{HS}}=\sum\limits_{i=1}^{\dim E_{\pi}}\left(
Se_{i},Te_{i}\right)  ,
\]
where $\left\{  e_{1},...,e_{\dim E_{\pi}}\right\}  $ is an orthonormal basis
of $E_{\pi}.$ It can be proved that the Hilbert-Schmidt \ inner product is
independent of the choice of the orthonormal basis of $E_{\pi}.$ We denote the
corresponding norm by $\left\Vert .\right\Vert _{\mathrm{HS}},$ that is, for
$T\in\mathrm{End}(E_{\pi})$, we put
\[
\left\Vert T\right\Vert _{\mathrm{HS}}^{2}=\left(  T,T\right)  _{\mathrm{HS}%
}=\sum\limits_{i=1}^{\dim E_{\pi}}\left\Vert Te_{i}\right\Vert _{\pi}^{2}%
\]
and
\[
\left\Vert \widehat{\mu_{a}^{\left(  p\right)  }}\right\Vert _{\mathrm{2}}%
^{2}=\sum\limits_{\left[  \pi\right]  \in\widehat{U}}\left(  \dim E_{\pi
}\right)  \left\Vert \widehat{\mu^{\left(  p\right)  }}\left(  \pi\right)
\right\Vert _{\mathrm{HS}}^{2}.
\]
For more details on the material of this section see \cite{bour}, or \cite{hr}.

The aim of this section is to prove the following result.

\begin{proposition}
\label{prop1} Let $p\geq2$. Then
\[
\left\Vert \widehat{\mu_{a}^{\left(  p\right)  }}\right\Vert _{\mathrm{2}}%
^{2}=\sum\limits_{\substack{\left[  \pi\right]  \in\widehat{U} \\\pi\text{
spherical}}}\left(  \dim E_{\pi}\right)  \left\vert \left(  \pi\left(
a\right)  X_{\pi},X_{\pi}\right)  \right\vert ^{2p}.
\]
where $X_{\pi}$ a basis for $E_{\pi}^{K}$ with $\left\Vert X_{\pi}\right\Vert
=1$.
\end{proposition}

In order to prove the proposition we first need four preparatory Lemmas.
Let
\[
C^{\#}\left(  U\right)  =\left\{  f\text{ in }C\left(  U\right)  \mid f\left(
k_{1}gk_{2}\right)  =f\left(  g\right)  \text{ for all }k_{1},k_{2}\text{ in
}K\text{ and }g\text{ in }U\right\}  .
\]
The pair $\left(  U,K\right)  $ is said to be a \emph{Gelfand pair} if the
algebra $C^{\#}\left(  U\right)  $, under convolution, is abelian.

\begin{lemma}
The space $E_{\pi}^{K}$ is at most one-dimensional, or equivalently,
\[
\dim E_{\pi}^{K}=0 \text{\hspace{0.5cm} or \hspace{0.5cm}}\dim E_{\pi}^{K}=1.
\]

\end{lemma}

\begin{proof}
Since $U/K$ is a symmetric space, the pair $\left(  U,K\right)  $ is a Gelfand
pair (see for example \cite{Wolf}, Corollary 8.1.4). The Lemma follows from
Proposition $6.3.1$ of \cite{Dijk} which says that $\left(  U,K\right)  $ is a
Gelfand pair if and only if the space $E_{\pi}^{K}$ of $K$-fixed vectors is at
most one-dimensional.
\end{proof}

\begin{lemma}
\label{lem-mu} If $\dim E_{\pi}^{K}=0,$ then $\widehat{\mu}_{a}\left(
\pi\right)  =0.$
\end{lemma}

\begin{proof}
From the definition of $\mu_{a}$, we get the equalities
\begin{align*}
\widehat{\mu_{a}}\left(  \pi\right)  \left(  X\right)   &  ={\int_{U}}%
\pi\left(  g^{-1}\right)  \left(  X\right)  \mathrm{d}\mu_{a}\left(  g\right)
\\
&  ={\int_{K\times K}}\pi\left(  \left(  k_{1}ak_{2}\right)  ^{-1}\right)
\left(  X\right)  \mathrm{d}k_{1}\mathrm{d}k_{2}\\
&  ={\int_{K\times K}}\pi\left(  k_{2}^{-1}\right)  \pi\left(  a^{-1}\right)
\pi\left(  k_{1}^{-1}\right)  \left(  X\right)  \mathrm{d}k_{1}\mathrm{d}%
k_{2}.
\end{align*}
Then we obtain
\begin{align*}
\pi\left(  k\right)  \widehat{\mu}_{a}\left(  \pi\right)  \left(  X\right)
&  =\pi\left(  k\right)  {\int_{U}}\pi\left(  g^{-1}\right)  \left(  X\right)
\mathrm{d}\mu_{a}\left(  g\right) \\
&  =\pi\left(  k\right)  {\int_{K\times K}}\pi\left(  k_{2}^{-1}\right)
\pi\left(  a^{-1}\right)  \pi\left(  k_{1}^{-1}\right)  \left(  X\right)
\mathrm{d}k_{1}\mathrm{d}k_{2}\\
&  =\pi\left(  k\right)  {\int_{K}}\pi\left(  k_{2}^{-1}\right)  \pi\left(
a^{-1}\right)  \left(  {\int_{K}}\pi\left(  k_{1}^{-1}\right)  \left(
X\right)  \mathrm{d}k_{1}\right)  \mathrm{d}k_{2}\\
&  ={\int_{K}}\pi\left(  \left(  k_{2}k^{-1}\right)  ^{-1}\right)  \pi\left(
a^{-1}\right)  \left(  {\int_{K}}\pi\left(  k_{1}^{-1}\right)  \left(
X\right)  \mathrm{d}k_{1}\right)  \mathrm{d}k_{2}\\
&  ={\int_{K}}\pi\left(  k_{2}^{-1}\right)  \pi\left(  a^{-1}\right)  \left(
{\int_{K}}\pi\left(  k_{1}^{-1}\right)  \left(  X\right)  \mathrm{d}%
k_{1}\right)  \mathrm{d}k_{2}\\
&  ={\int_{K\times K}}\pi\left(  k_{2}^{-1}\right)  \pi\left(  a^{-1}\right)
\pi\left(  k_{1}^{-1}\right)  \left(  X\right)  \mathrm{d}k_{1}\mathrm{d}%
k_{2},\\
&  =\widehat{\mu}_{a}\left(  \pi\right)  \left(  X\right)  .
\end{align*}
Therefore $\widehat{\mu}_{a}\left(  \pi\right)  \left(  X\right)  \in E_{\pi
}^{K}.$ The Lemma follows from the assumption $E_{\pi}^{K}=\left\{  0\right\}
. $
\end{proof}

\begin{lemma}
\label{lem234} 
Let $\left(  \pi,E_{\pi}\right)  $ be a spherical
representation, i.e., $\dim E_{\pi}^{K}=1$, $X_{\pi}$ be a basis for $E_{\pi
}^{K}$ with $\left\Vert X_{\pi}\right\Vert =1$ and let $\left\{  X_{1}%
,\dots,X_{d_{\pi}}\right\}  $ be an orthonormal basis of $E_{\pi}$ with
$X_{1}=X_{\pi}.$ Then for every positive integer $k$,
\[
\left(  \widehat{\mu_{a}^{k}}\left(  \pi\right)  X_{i},X_{j}\right)  =0\text{
for all pairs }\left(  i,j\right)  \neq\left(  1,1\right)  .
\]

\end{lemma}

\begin{proof}
From the properties of the Fourier Transform, we have
\[
\widehat{\mu_{a}^{k}}\left(  \pi\right)  =\widehat{\mu}_{a}\left(  \pi\right)
...\widehat{\mu}_{a}\left(  \pi\right)  \text{ \ \ (}k\text{ times).}%
\]
As in the proof of Lemma \ref{lem-mu}, we have%
\[
\pi\left(  k\right)  \widehat{\mu}_{a}\left(  \pi\right)  \left(  X\right)
=\widehat{\mu}_{a}\left(  \pi\right)  \left(  X\right)  \text{ for all
}k\text{ in }K\text{ and all }X\in E_{\pi},
\]
i.e.,
\[
\widehat{\mu}_{a}\left(  \pi\right)  \left(  X\right)  \in E_{\pi}%
^{K}=\mathsf{span}\left\{  X_{1}\right\}  \text{ for all }X\text{ in }E_{\pi
}\text{.}%
\]

Thus
\[
\widehat{\mu_{a}^{k}}\left(  \pi\right)  \left(  X\right)  \in E_{\pi}%
^{K}=\mathsf{span}\left\{  X_{1}\right\}  \text{ for all }X\text{ in }E_{\pi
}.
\]
Let $k$ in $K$, $X$ in $E_{\pi}$ and let $P\left(  X\right)  ={\int_{K}}%
\pi\left(  h^{-1}\right)  \left(  X\right)  \mathrm{d}h$. Then
\begin{align*}
\pi\left(  k\right)  P\left(  X\right)   &  ={\int_{K}}\pi\left(  k\right)
\pi\left(  h^{-1}\right)  \left(  X\right)  \mathrm{d}h\\
&  ={\int_{K}}\pi\left(  \left(  hk^{-1}\right)  ^{-1}\right)  \left(
X\right)  \mathrm{d}h\\
&  =P\left(  X\right)  .
\end{align*}
Hence
\[
P\left(  X\right)  \in E_{\pi}^{K}\text{ for all }X\in E_{\pi}.
\]
From $P\left(  X_{i}\right)  =0$ for $i\geq2$, we deduce that
\begin{align*}
\widehat{\mu_{a}^{k}}\left(  \pi\right)  \left(  X_{i}\right)   &
=\widehat{\mu_{a}^{k-1}}\left(  \pi\right)  \left(  \widehat{\mu}_{a}\left(
\pi\right)  \left(  X_{i}\right)  \right) \\
&  =\widehat{\mu_{a}^{k-1}}\left(  \pi\right)  \left(  {\textstyle\int_{K}}
\pi\left(  k_{2}^{-1}\right)  \pi\left(  a^{-1}\right)  P\left(  X_{i}\right)
\mathrm{d}k_{2}\right) \\
&  =0
\end{align*}
since $P\left(  X_{i}\right)  =0$ for $i \geq2$. Therefore we obtain
\[
\left(  \widehat{\mu_{a}^{k}}\left(  \pi\right)  X_{i},X_{j}\right)  =0\text{
for all }i\text{ and for all }j\neq1.
\]
Hence
\[
\left(  \widehat{\mu_{a}^{k}}\left(  \pi\right)  X_{i},X_{j}\right)  =0\text{
for all }\left(  i,j\right)  \neq\left(  1,1\right)  .
\]
\end{proof}

Here is our final preparatory lemma.

\begin{lemma}
\label{lemrrt} With the preceding notation, one has
\[
\left(  \widehat{\mu_{a}^{p}}\left(  \pi\right)  X_{1}, X_{1}\right)  =
\left(  \pi\left(  a^{-1}\right)  X_{1},X_{1}\right)  ^{p}.
\]

\end{lemma}

\begin{proof}
From the proof of Lemma \ref{lem-mu}, we have%
\begin{align*}
\widehat{\mu_{a}}\left(  \pi\right)  X_{1}  &  ={\int_{K}}\pi\left(
k_{2}^{-1}\right)  \pi\left(  a^{-1}\right)  \left(  {\int_{K}}\pi\left(
k_{1}^{-1}\right)  \left(  X_{1}\right)  \mathrm{d}k_{1}\right)
\mathrm{d}k_{2}\\
&  ={\int_{K}}\pi\left(  k_{2}^{-1}\right)  \pi\left(  a^{-1}\right)  P\left(
X_{1}\right)  \mathrm{d}k_{2}\\
&  ={\int_{K}}\pi\left(  k_{2}^{-1}\right)  \pi\left(  a^{-1}\right)  X_{1}
\mathrm{d}k_{2}
\end{align*}
since $P\left(  X_{1}\right)  =X_{1}$ as $X_{1}\in E_{\pi}^{K}$. It follows
$$
\widehat{\mu_{a}}\left(  \pi\right)  X_{1} =P\left(  \pi\left(  a^{-1}\right)  X_{1}\right)  .
$$
Hence%
\begin{align}
\left(  \widehat{\mu_{a}}\left(  \pi\right)  X_{1}, X_{1}\right)   &  =\left(
P\left(  \pi\left(  a^{-1}\right)  X_{1}\right)  , X_{1}\right) \nonumber\\
&  =\left(  \pi\left(  a^{-1}\right)  X_{1},P^{\ast}\left(  X_{1}\right)
\right) \nonumber\\
&  =\left(  \pi\left(  a^{-1}\right)  X_{1},X_{1}\right)  , \label{dsdsf}%
\end{align}
the last equality following from the fact that $P$ is the projection on
$\mathsf{span}\left\{  X_{1}\right\}  ,$ hence $P^{\ast}=P.$

By Lemma \ref{lem234}, in the orthonormal basis $\left\{  X_{1},...,X_{d_{\pi
}}\right\}  $ of $E_{\pi},$ we have
\[
\left(  \widehat{\mu_{a}^{p}}\left(  \pi\right)  X_{i},X_{j}\right)  =0\text{
for all }\left(  i,j\right)  \neq\left(  1,1\right)  ,
\]
hence the matrix associated to the endomorphism $\widehat{\mu_{a}^{p}}\left(
\pi\right)  ,$ denoted by $M\left(  \widehat{\mu_{a}^{p}}\left(  \pi\right)
\right)  ,$ is given by%
\[
M\left(  \widehat{\mu_{a}^{p}}\left(  \pi\right)  \right)  =\left(
\begin{array}
[c]{cccc}%
\left(  \widehat{\mu_{a}^{p}}\left(  \pi\right)  X_{1},X_{1}\right)  & 0 &
... & 0\\
0 & 0 & \dots & 0\\
\vdots & \vdots &  & \vdots\\
0 & 0 & ... & 0
\end{array}
\right)  .
\]
Since
\[
\widehat{\mu_{a}^{p}}\left(  \pi\right)  =\widehat{\mu_{a}}\left(  \pi\right)
...\widehat{\mu_{a}}\left(  \pi\right)  \text{ ($p$ times)},
\]
we infer
\[
M\left(  \widehat{\mu_{a}^{p}}\left(  \pi\right)  \right)  = \left(
\begin{array}
[c]{cccc}%
\left(  \widehat{\mu_{a}} \left(  \pi\right)  X_{1}, X_{1} \right)  ^{p} & 0 &
... & 0\\
0 & 0 & \dots & 0\\
\vdots & \vdots &  & \vdots\\
0 & 0 & ... & 0
\end{array}
\right)  .
\]
Hence
\begin{align*}
\left(  \widehat{\mu_{a}^{p}}\left(  \pi\right)  X_{1}, X_{1}\right)   &
=\left(  \widehat{\mu_{a}}\left(  \pi\right)  X_{1}, X_{1}\right)  ^{p}\\
&  = \left(  \pi\left(  a^{-1}\right)  X_{1}, X_{1}\right)  ^{p}%
\end{align*}
using \eqref{dsdsf}, which proves our claim.
\end{proof}

We are now ready to give the proof of our main proposition of this section.

\begin{proof}
[Proof of Proposition \ref{prop1}]We use the notations of Lemma \ref{lem234}.
One has
\begin{eqnarray*}
\left\Vert \widehat{\mu^{\left(  p\right)  }}\left(  \pi\right)  \right\Vert
_{\mathrm{HS}}^{2}  &  = & \sum\limits_{i=1}^{\dim E_{\pi}}\left\Vert
\widehat{\mu_{a}^{p}}\left(  \pi\right)  X_{i}\right\Vert ^{2}  \\
&  = &\sum\limits_{i=1}^{\dim E_{\pi}}\sum\limits_{j=1}^{\dim E_{\pi}}\left\vert
\left(  \widehat{\mu_{a}^{p}}\left(  \pi\right)  X_{i},X_{j}\right)
\right\vert ^{2} \\
&  =& \left\vert \left(  \widehat{\mu_{a}^{p}}\left(  \pi\right)  X_{1},X_{1}\right)  \right\vert ^{2}\\
&  = & \left\vert \left(  \pi\left(  a^{-1}\right)  X_{1},X_{1}\right)
\right\vert ^{2p}
\end{eqnarray*}
by Lemma \ref{lem234} and then Lemma \ref{lemrrt}. Using the fact that $\pi$ is unitary, we deduce that
$$
\left\Vert \widehat{\mu^{\left(  p\right)  }}\left(  \pi\right)  \right\Vert
_{\mathrm{HS}}^{2}
 =\left\vert \left(  \pi\left(  a\right)  X_{1},X_{1}\right)  \right\vert
^{2p}
$$
Hence
\begin{align*}
\left\Vert \widehat{\mu_{a}^{\left(  p\right)  }}\right\Vert _{\mathrm{2}%
}^{2}  &  =\sum\limits_{\substack{\left[  \pi\right]  \in\widehat{U}
\\\pi\text{ spherical}}}\left(  \dim E_{\pi}\right)  \left\Vert \widehat{\mu
^{\left(  p\right)  }}\left(  \pi\right)  \right\Vert _{\mathrm{HS}}^{2}\\
&  =\sum\limits_{\substack{\left[  \pi\right]  \in\widehat{U} \\\pi\text{
spherical}}}\left(  \dim E_{\pi}\right)  \left\vert \left(  \pi\left(
a\right)  X_{\pi},X_{\pi}\right)  \right\vert ^{2p},
\end{align*}
as announced in the statement of the Proposition.
\end{proof}

\section{The Case of $SU\left(  2\right)  /SO\left(  2\right)  $}
\label{sec4}

Let
\[
E_{n}=\mathsf{span}\left\{  z_{1}^{k}z_{2}^{n-k}\mid0\leq k\leq n\right\}
\]
the complex vector space of homogeneous polynomials in two variables. There is
a natural action of $SU\left(  2\right)  $ on $E_{n}$ as follows:
\[%
\begin{array}
[c]{cccc}%
\pi_{n}: & SU\left(  2\right)  & \longrightarrow & GL\left(  E_{n}\right) \\
& A & \longmapsto & \pi_{n}\left(  A\right)
\end{array}
\]
where
\[
\pi_{n}\left(  A\right)  \left(  P\left(  z_{1},z_{2}\right)  \right)
=P\left(  \left(  z_{1},z_{2}\right)  A\right)  .
\]
It can be shown that $\pi_{n}$ is irreducible and every irreducible
representation of $SU\left(  2\right)  $ is of that form (see \cite{Wolf},
Proposition 5.7.5). Therefore $\widehat{SU\left(  2\right)  }\simeq\mathbb{N}$
and the dimension of the representation corresponding to $n$ is $n+1.$

Consider the inner product in $E_{n}$ defined as follows:
\[
\left(  \sum\limits_{i=0}^{n}a_{i}z_{1}^{i}z_{2}^{n-i},\sum\limits_{j=0}%
^{n}b_{j}z_{1}^{j}z_{2}^{n-j}\right)  =\sum\limits_{k=0}^{n}k!\left(
n-k\right)  !a_{k}\overline{b}_{k}.
\]
With this product, the representation $\pi_{n}:SU\left(  2\right)  \rightarrow
GL\left(  E_{n}\right)  $ is unitary. Denote by $\left\vert \left\vert
.\right\vert \right\vert $ the corresponding norm.

The vector $X_{\pi_{2n}}=\left(  z_{1}^{2}+z_{2}^{2}\right)  ^{n}$ is
invariant under the action of $SO\left(  2\right)  .$ Hence $\left(  \pi
_{2n},E_{2n}\right)  $ is of class one. Conversely, every class one
irreducible representation of $SU\left(  2\right)  $ is of this form. For
simplicity, we put $X_{\pi_{2n}}=X_{\pi}$.

From
\[
X_{\pi}=\left(  z_{1}^{2}+z_{2}^{2}\right)  ^{n}=\sum\limits_{k=0}^{n}%
\binom{n}{k}z_{1}^{2k}z_{2}^{2\left(  n-k\right)  },
\]
we get
\begin{align*}
\left\vert \left\vert X_{\pi}\right\vert \right\vert ^{2}  &  =\left(
\sum\limits_{k=0}^{n}\binom{n}{k}z_{1}^{2k}z_{2}^{2\left(  n-k\right)  }%
,\sum\limits_{k=0}^{n}\binom{n}{k}z_{1}^{2k}z_{2}^{2\left(  n-k\right)
}\right) \\
&  =\sum\limits_{k=0}^{n}\left(  2k\right)  !\left(  2n-2k\right)  !\binom
{n}{k}^{2}.
\end{align*}
Let
\[
\widetilde{X}_{\pi}=\frac{X_{\pi}}{\left\vert \left\vert X_{\pi}\right\vert
\right\vert }\quad\text{ and }\quad a_{\vartheta}=\left(
\begin{array}
[c]{cc}%
e^{i\vartheta} & 0\\
0 & e^{-i\vartheta}%
\end{array}
\right)
\]
where $a$ is not in the normalizer of $SO\left(  2\right)  $ in $SU\left(
2\right)  .$

From now on, we shall use the notation already introduced $U=SU\left(
2\right)  $ and $K=SO\left(  2\right)  $. The proof of the following
proposition is straightforward.

\begin{proposition}
\label{prop12} 
Let $A=\exp\left(  i\mathfrak{a}\right)  $, $g=k_{1}ak_{2}\in
U, $ $k_{1},k_{2}\in K,$ $a\in A,$ and denote by $N_{A}\left(  K\right)  $ the
normalizer of $K$ in $A.$ Then

\begin{enumerate}
\item[(i)] $g\in N_{U}\left(  K\right)  $ if and only of $a\in N_{A}\left(
K\right)  $,

\item[(ii)] the normalizer $N_{A}(K)$ satisfies
\[
N_{A}\left(  K\right)  =\left\{  a_{\vartheta}=\left(
\begin{array}
[c]{cc}%
e^{i\vartheta} & 0\\
0 & e^{-i\vartheta}%
\end{array}
\right)  \mid e^{4i\vartheta}=1\right\}  .
\]

\end{enumerate}
\end{proposition}

Therefore the element $a_{\pi/2}=\left(
\begin{array}
[c]{cc}%
i & 0\\
0 & -i
\end{array}
\right)  $ is clearly in the normalizer of $K$ in $U$ and for all
$\vartheta\in\left(  0,\frac{\pi}{2}\right)  ,$ $a_{\theta}\notin N_{A}\left(
K\right)  .$ In the light of that fact, in all what follows, we will assume
that $\vartheta\in\left(  0,\frac{\pi}{2}\right)  .$

Put%
\[
\varphi_{2n}\left(  a_{\vartheta}\right)  :=\left(  \pi_{2n}\left(
a_{\vartheta}\right)  \widetilde{X}_{\pi},\widetilde{X}_{\pi}\right)  .
\]
Since
\begin{align*}
\left(  \pi_{2n}\left(  a_{\vartheta}\right)  \widetilde{X}_{\pi}\right)
\left(  z\right)   &  =\frac{1}{\left\vert \left\vert X_{\pi}\right\vert
\right\vert }\sum\limits_{k=0}^{n}\binom{n}{k}\pi_{2n}\left(  a_{\vartheta
}\right)  \left(  z_{1}^{2k}z_{2}^{2n-2k}\right) \\
& \\
&  =\frac{1}{\left\vert \left\vert X_{\pi}\right\vert \right\vert }%
\sum\limits_{k=0}^{n}\binom{n}{k}e^{i\vartheta\left(  4k-2n\right)  }%
z_{1}^{2k}z_{2}^{2n-2k},
\end{align*}
we deduce that
\begin{align*}
\varphi_{2n}\left(  a_{\vartheta}\right)   &  =\frac{1}{\left\vert \left\vert
X_{\pi}\right\vert \right\vert ^{2}}\left(  \sum\limits_{k=0}^{n}\binom{n}%
{k}e^{i\vartheta\left(  4k-2n\right)  }z_{1}^{2k}z_{2}^{2n-2k},\sum
\limits_{k=0}^{n}\binom{n}{k}z_{1}^{2k}z_{2}^{2\left(  n-k\right)  }\right) \\
& \\
&  =\frac{\sum\limits_{k=0}^{n}\left(  2k\right)  !\left(  2n-2k\right)
!\binom{n}{k}^{2}e^{i\vartheta\left(  4k-2n\right)  }}{\sum\limits_{k=0}%
^{n}\left(  2k\right)  !\left(  2n-2k\right)  !\binom{n}{k}^{2}}.
\end{align*}
Therefore, applying Proposition \ref{prop1}, we get
\begin{align}
\left\Vert \widehat{\mu}_{a_{\vartheta}}^{\left(  p\right)  }\right\Vert _{2}^2
&  ={\sum\limits_{n=1}^{\infty}}\left(  2n+1\right)  \left\vert \varphi
_{2n}\left(  a_{\vartheta}\right)  \right\vert ^{2p}\nonumber\\
&  ={\sum\limits_{n=1}^{\infty}}\left(  2n+1\right)  \left\vert \frac
{\sum\limits_{k=0}^{n}\left(  2k\right)  !\left(  2n-2k\right)  !\binom{n}%
{k}^{2}e^{i\vartheta\left(  4k-2n\right)  }}{\sum\limits_{k=0}^{n}\left(
2k\right)  !\left(  2n-2k\right)  !\binom{n}{k}^{2}}\right\vert ^{2p}%
\nonumber\\
&  =\sum\limits_{n=1}^{\infty}\left(  2n+1\right)  \left\vert \frac
{\sum\limits_{k=0}^{n}\left(  2k\right)  !\left(  2n-2k\right)  !\binom{n}%
{k}^{2}e^{4ik\vartheta}}{\sum\limits_{k=0}^{n}\left(  2k\right)  !\left(
2n-2k\right)  !\binom{n}{k}^{2}}\right\vert ^{2p}. \label{eq2}%
\end{align}

\section{Proof of The Main Theorem $(i)$}
\label{sec5}

To study the convergence of the series \eqref{eq2}, we first need the
following easy lemma, which will elucidate the behavior of its denominator.

\begin{lemma}
\label{lem1} 
One has
\[
\sum_{k=0}^{n}\binom{2k}{k}\binom{2n-2k}{n-k}=4^{n}.
\]

\end{lemma}

\begin{proof}
We use the method of generating functions. The generating function of the
sequence $\binom{2n}{n}$ is the function $\sqrt{1-4x}$, i.e. (say for a real
number $x$ such that $|x|< 1/4$)
\[
\frac{1}{\sqrt{1-4x}}={\sum\limits_{n=0}^{\infty}}\binom{2n}{n}x^{n}.
\]
Hence, by squaring,
\[
\frac{1}{1-4x} =\left(  {\sum\limits_{k=0}^{\infty}}\binom{2k}{k}
x^{k}\right)  \left(  {\sum\limits_{l=0}^{\infty}}\binom{2l}{l}x^{l}\right)
={\sum\limits_{n=0}^{\infty}}\left(  {\sum\limits_{k=0}^{n}}\binom{2k}
{k}\binom{2n-2k}{n-k}\right)  x^{n}.
\]
The Lemma follows from the identity
\[
\frac{1}{1-4x}={\sum\limits_{n=0}^{\infty}}4^{n}x^{n}.
\]

\end{proof}

To study the numerator appearing in the terms of the series \eqref{eq2}, we
introduce for a real number $\vartheta$ (non-zero modulo $\pi/2$) and any
integer $n\geq1$, the sum
\begin{equation}
t_{n}\left(  \vartheta\right)  =\sum\limits_{k=0}^{n}\binom{2k}{k}%
\binom{2n-2k}{n-k}\exp\left(  4ik\vartheta\right)  \label{eq01}%
\end{equation}
and investigate this quantity. Note first that $t_{n}$ is periodic with period
$\pi/2$.

\begin{lemma}
\label{lem2} 
Let $\vartheta$ be a non-zero (modulo $\pi/2$) real number.
There is a positive real number $C\left(  \vartheta\right)  $ such that for
any positive integer $n$, one has
\[
\left\vert t_{n}\left(  \vartheta\right)  \right\vert \leq C\left(
\vartheta\right)  \frac{4^{n}}{\sqrt{n}}.
\]
\end{lemma}

\begin{proof}
We distinguish two cases depending on the parity of $n$. 
\medskip

\noindent Case a) : Let us assume first that $n$ is odd. The sum
\[
t_{n}\left(  \vartheta\right)  =\sum\limits_{k=0}^{n}\binom{2k}{k}%
\binom{2n-2k}{n-k}\exp\left(  4ik\vartheta\right)
\]
can be written as $t_{n}\left(  \vartheta\right)  =t_{n}^{(1)}\left(
\vartheta\right)  +t_{n}^{(2)}\left(  \vartheta\right)  $, where
\[
t_{n}^{(1)}\left(  \vartheta\right)  =\sum\limits_{k=0}^{(n-1)/2}\binom{2k}%
{k}\binom{2n-2k}{n-k}\exp\left(  4ik\vartheta\right)
\]
and
\[
t_{n}^{(2)}\left(  \vartheta\right)  =\sum\limits_{k=(n+1)/2}^{n}\binom{2k}%
{k}\binom{2n-2k}{n-k}\exp\left(  4ik\vartheta\right)  .
\]
By making the change of variable $j=n-k$ in $t_{n}^{(2)}(\vartheta),$ we get
\[
t_{n}^{(2)}\left(  \vartheta\right)  =\exp\left(  4in\vartheta\right)
t_{n}^{(1)}\left(  -\vartheta\right)
\]
which yields
\begin{equation}
\label{gggggg}t_{n}\left(  \vartheta\right)  =t_{n}^{(1)}\left(
\vartheta\right)  +\exp\left(  4in\vartheta\right)  t_{n}^{(1)}\left(
-\vartheta\right)  .
\end{equation}
We therefore restrict the study of $t_{n}\left(  \vartheta\right)  $ to the
one of $t_{n}^{(1)}\left(  \vartheta\right)  $. Put
\[
t_{n}^{(1)}\left(  \vartheta\right)  =\sum\limits_{k=0}^{\left(  n-1\right)
/2}v_{k}\exp\left(  4ik\vartheta\right)
\]
where
\[
v_{k}=\binom{2k}{k}\binom{2n-2k}{n-k}.
\]
Put $\ u_{-1}\left(  \vartheta\right)  =0$ and let $u_{k}\left(
\vartheta\right)  ={\sum\limits_{j=0}^{k}}\exp\left(  4ij\vartheta\right)  $
for $k\geq0$.
Then
\[
u_{k}\left(  \vartheta\right)  =\frac{\exp\left(  4i\left(  k+1\right)
\vartheta\right)  -1}{\exp\left(  4i\vartheta\right)  -1}=\left(  \frac
{\sin\left(  2\left(  k+1\right)  \vartheta\right)  }{\sin2\vartheta}\right)
\exp\left(  2ik\vartheta\right)
\]
from which we get
\[
\left\vert u_{k}\left(  \vartheta\right)  \right\vert =\left\vert \frac
{\sin\left(  2\left(  k+1\right)  \vartheta\right)  }{\sin2\vartheta
}\right\vert \leq\frac{1}{\left\vert \sin2\vartheta\right\vert }=c_{0}\left(
\vartheta\right)  .
\]
Abel's transformation gives
\begin{align*}
t_{n}^{(1)}\left(  \vartheta\right)   &  ={\sum\limits_{k=0}^{(n-1)/2}}%
v_{k}\left(  u_{k}\left(  \vartheta\right)  -u_{k-1}\left(  \vartheta\right)
\right) \\
&  ={\sum\limits_{k=0}^{(n-1)/2}}v_{k}u_{k}\left(  \vartheta\right)
-{\sum\limits_{k=-1}^{(n-1)/2-1}}v_{k+1}u_{k}\left(  \vartheta\right) \\
&  ={\sum\limits_{k=0}^{(n-1)/2-1}}\left(  v_{k}-v_{k+1}\right)  u_{k}\left(
\vartheta\right)  +u_{(n-1)/2}\left(  \vartheta\right)  v_{(n-1)/2}.
\end{align*}
We now observe that the sequence $\left(  v_{k}\right)  _{k\geq0}$ is
decreasing for $k<n/2$. Indeed, it is immediate to compute that
\[
v_{k}-v_{k+1}=2\binom{2k}{k}\binom{2(n-k-1)}{n-k-1}\left(  \frac
{n-2k-1}{(n-k)(k+1)}\right)  >0.
\]
Hence
\begin{align*}
\left\vert t_{n}^{(1)}\left(  \vartheta\right)  \right\vert  &  \leq
\sum\limits_{k=0}^{(n-1)/2-1}\left(  v_{k}-v_{k+1}\right)  \left\vert
u_{k}\left(  \vartheta\right)  \right\vert +\left\vert u_{(n-1)/2}\left(
\vartheta\right)  v_{(n-1)/2}\right\vert \\
&  \leq c_{0} \left(  \vartheta\right)  \left(  \sum\limits_{k=0}%
^{(n-1)/2-1}\left(  v_{k}-v_{k+1}\right)  +v_{(n-1)/2}\right) \\
&  =c_{0}\left(  \vartheta\right)  \binom{2n}{n}\\
&  \sim c_{0}\left(  \vartheta\right)  \frac{4^{n}}{\sqrt{\pi n}}%
\end{align*}
by Stirling's formula. Thus there is a constant $c_{1}(\vartheta)$ such that,
for all odd integers $n$,
\[
\left\vert t_{n}^{(1)}\left(  \vartheta\right)  \right\vert \leq c_{1}\left(
\vartheta\right)  \frac{4^{n}}{\sqrt{n}}.
\]
It follows, by \eqref{gggggg}, that
\begin{align*}
\left\vert t_{n}\left(  \vartheta\right)  \right\vert  &  =\left\vert
t_{n}^{(1)}\left(  \vartheta\right)  +\exp\left(  4in\vartheta\right)
t_{n}^{(1)}\left(  -\vartheta\right)  \right\vert \\
&  \leq\left\vert t_{n}^{(1)}\left(  \vartheta\right)  \right\vert +\left\vert
t_{n}^{(1)}\left(  -\vartheta\right)  \right\vert \\
&  \leq(c_{1}\left(  \vartheta\right)  +c_{1}\left(  -\vartheta\right)
)\frac{4^{n}}{\sqrt{n}}.
\end{align*}
Hence the Lemma with $C\left(  \vartheta\right)  =c_{1}\left(  \vartheta
\right)  +c_{1}\left(  -\vartheta\right)  $ in the case $n$ is odd.
\medskip

\noindent Case b) : Suppose now that $n$ is even, and put
\[
t_{n}^{(1)}\left(  \vartheta\right)  =\sum\limits_{k=0}^{\left(  n/2\right)
-1}\binom{2k}{k}\binom{2n-2k}{n-k}\exp\left(  4ik\vartheta\right)
\]
and
\[
t_{n}^{(2)}\left(  \vartheta\right)  =\sum\limits_{k=\left(  n/2\right)
+1}^{n}\binom{2k}{k}\binom{2n-2k}{n-k}\exp\left(  4ik\vartheta\right)  .
\]
By making the change of variable $j=n-k$ in $t_{n}^{(2)}(\vartheta),$ we get
\[
t_{n}^{(2)}\left(  \vartheta\right)  =\exp\left(  4in\vartheta\right)
t_{n}^{(1)}\left(  -\vartheta\right)
\]
which yields
\[
t_{n}\left(  \vartheta\right)  =t_{n}^{(1)}\left(  \vartheta\right)
+\exp\left(  4in\vartheta\right)  t_{n}^{(1)}\left(  -\vartheta\right)
+\binom{n}{n/2}^{2}\exp\left(  2in\vartheta\right)  .
\]
Using the same argument as above and the fact, due again to Stirling's
formula, that
\[
\binom{n}{n/2}^{2}\exp\left(  2in\vartheta\right)  =O\left(  \frac{4^{n}}%
{n}\right)  ,
\]
we deduce the Lemma.
\end{proof}

We are now able to deduce the key-result of this section.

\begin{proposition}
\label{prop4}Let $\vartheta$ be a non-zero (modulo $\pi/2$) real number. If
$p>2$, then the series
\begin{equation}
{\displaystyle\sum\limits_{n=1}^{\infty}}\left(  2n+1\right)  \left\vert
\frac{\sum\limits_{k=0}^{n}\left(  2k\right)  !\left(  2n-2k\right)
!\binom{n}{k}^{2}\exp(4ik\vartheta)}{\sum\limits_{k=0}^{n}\left(  2k\right)
!\left(  2n-2k\right)  !\binom{n}{k}^{2}}\right\vert ^{2p} \label{series1}%
\end{equation}
converges.
\end{proposition}

\begin{proof}
If we put
\[
\mathcal{S}_{n}\left(  \vartheta\right)  ={\textstyle\sum\limits_{k=0}^{n}%
}\left(  2k\right)  !\left(  2n-2k\right)  !\binom{n}{k}^{2}\exp\left(
4ik\vartheta\right)
\]
then the series $\left(  \ref{series1}\right)  $ is equal to%
\[
{\textstyle\sum\limits_{n=1}^{\infty}}\left(  2n+1\right)  \left\vert
\frac{\mathcal{S}_{n}\left(  \vartheta\right)  }{\mathcal{S}_{n}\left(
0\right)  }\right\vert ^{2p}.
\]
But
\begin{align*}
\mathcal{S}_{n} ( \vartheta)  &  = \sum_{k=0}^{n} (2k)! (2n-2k)! \frac{n!^{2}%
}{( k! (n-k)!)^{2}} \exp( 4ik\vartheta)\\
&  = \sum_{k=0}^{n} n!^{2} \left(  \frac{(2k)!}{ k!^{2}}\right)  \left(
\frac{(2n-2k)!}{(n-k)! ^{2}}\right)  \exp\left(  4ik\vartheta\right) \\
&  = n!^{2} \sum_{k=0}^{n} \binom{2k}{k}\binom{2n-2k}{n-k} \exp\left(
4ik\vartheta\right)  .
\end{align*}
We obtain that the series \eqref{series1} is%
\begin{equation}
\sum_{n=1}^{\infty}(2n+1)\left\vert \frac{\sum_{k=0}^{n}\binom{2k}{k}%
\binom{2n-2k}{n-k}\exp\left(  4ik\vartheta\right)  }{\sum_{k=0}^{n}\binom
{2k}{k}\binom{2n-2k}{n-k}}\right\vert ^{2p}=\sum_{n=1}^{\infty}%
(2n+1)\left\vert \frac{\sum_{k=0}^{n}\binom{2k}{k}\binom{2n-2k}{n-k}%
\exp\left(  4ik\vartheta\right)  }{4^{n}}\right\vert ^{2p} \label{eq02}%
\end{equation}
using Lemma \ref{lem1}. Lemma \ref{lem2}
then implies that
\[
\left(  2n+1\right)  \left\vert \frac{\sum\limits_{k=0}^{n}\binom{2k}{k}%
\binom{2n-2k}{n-k}\exp\left(  4ik\vartheta\right)  }{4^{n}}\right\vert
^{2p}\leq\left(  2n+1\right)  \left\vert \frac{C(\vartheta)}{\sqrt{n}%
}\right\vert ^{2p}\ll_{\vartheta}\frac{1}{n^{p-1}}.
\]
Therefore the series in $\left(  \ref{series1}\right)  $ converges as soon as
$p>2.$
\end{proof}

The proof of Theorem \ref{Main} (i) follows now easily.

\begin{proof}
[Proof of the Theorem \ref{Main} (i)]Combining Proposition \ref{prop1},
equation \eqref{eq2} and Proposition \ref{prop4}, we get
\begin{equation}
\left\Vert \widehat{\mu_{a}^{\left(  p\right)  }}\right\Vert _{\mathrm{2}}%
^{2}={\sum\limits_{\left[  \pi\right]  \in\widehat{U}}}\left(  \dim E_{\pi
}\right)  \left\Vert \widehat{\mu_{a}^{(p)}}\left(  \pi\right)  \right\Vert
_{\mathrm{HS}}^{2}<\infty. \label{plancherel}%
\end{equation}

The Plancherel isomorphism Theorem (see \cite{hr}, Theorem $28.43$) guarantees
the existence of a function $f_{a,p}\in L^{2}\left(  U\right)  $ such that
$\widehat{f_{a,p}}=\widehat{\mu_{a}^{(p)}}$. Hence $\mu_{a}^{(p)}=f_{a,p}
\mathrm{d}m_{U}$.
\end{proof}

\section{Proof of the Main Theorem $(ii)$ :\\ A Counter-Example to The
Dichotomy Conjecture}
\label{sec6}

In this section we will study the behavior of $t_{n}^{(1)}(\vartheta)$
introduced in the preceding section and find a lower bound for $t_{n}\left(
\vartheta\right)  $, valid for a dense enough set of indices. In this respect,
we need several preliminary results (having a number-theoretical flavour) that
we establish now.

First we recall the notion of \emph{lower density} of a set of positive
integers. The lower density of a set $A$ of integers is by definition
\[
\underline{\mathrm{d}} A = \liminf_{n \to+ \infty} \frac{|A \cap\{1,\dots,
n\}| }{n}.
\]
If in this definition, we can replace the $\liminf$ by a simple $\lim$ then we
simply speak of a \emph{density}.

Finally, if $u \in{\mathbb{R}}$, we shall also denote $A+ u =\{ a+u \text{ for
} a \in A \}$, a translate of $A$ and $u \cdot A = \{ ua \text{ for } a \in A
\}$, a dilate of $A$.

The following lemma is crucial for our purpose.

\begin{lemma}
\label{harmo} 
Let $A$ be a set of integers such that $\underline{\mathrm{d}} A >0$, then the series
\[
\sum_{a \in A} \frac{1}{a}
\]
diverges.
\end{lemma}

\begin{proof}
Write $n_{0} =0$ and $\alpha=\underline{\mathrm{d}} A >0$. By definition,
there is an integer $n_{1}$ such that (for instance)
\[
|A \cap\{1,\dots, n_{1}\}| > \left[  \frac{\alpha n_{1}}{2} \right]  +1
\]
holds. More generally, we may construct a sequence of integer $(n_{i})$ such
that
\begin{align*}
|A \cap\{n_{i-1}+1,\dots, n_{i}\}|  &  = |A \cap\{1,\dots, n_{i}\}| - |A
\cap\{1,\dots, n_{i-1}\}|\\
&  > \frac{2\alpha n_{i}}{3}- n_{i-1}\\
&  > \left[  \frac{\alpha n_{i}}{2} \right]  +1.
\end{align*}
But then
\begin{align*}
\sum_{a \in A\cap\{n_{i-1}+1,\dots, n_{i}\}} \frac{1}{a}  &  \geq\sum_{a=
n_{i} - [\alpha n_{i} /2]}^{n_{i}} \frac{1}{a}\\
&  > \int_{n_{i} - [\alpha n_{i} /2]}^{n_{i} +1} \frac{\mathrm{d}x}{x}\\
&  = \log\left(  \frac{n_{i} +1}{n_{i} - [\alpha n_{i} /2]} \right) \\
&  \sim- \log(1-\alpha/2) >0
\end{align*}
when $i$ tends to infinity. Summing these contributions, we get as $i$ tends
to $+\infty$,
\[
\sum_{a \in A, a \leq n_{i}} \frac1a \gtrsim- i \log(1-\alpha/2)
\]
and the series thus diverges.
\end{proof}

We shall also need the following definition. For a $\omega\in\left(
0,\pi\right)  $ and $c>1/2$, we define ${\mathcal{E}}_{\omega,c}$ as the set
of integers $n$ such that
\[
\sin\omega\sin n\omega\geq c.
\]
Let now
\[
{\mathcal{M}}=\left\{  \omega\in{\mathbb{R}}\text{ such that there exists
}c>\frac{1}{2}\text{ such that }\underline{\mathrm{d}}{\mathcal{E}}_{\omega
,c}>0\right\}  .
\]
Note immediately that ${\mathcal{M}}$ is not empty, since for instance one can
easily see that $\pi/3$ or $\pi/4$ are in ${\mathcal{M}}$. In fact, it is much
bigger as shown by the following lemma for which we recall the following
definition : a sequence of real numbers $(x_{n})$ is said to be uniformly
distributed modulo 1 if for any $0\leq a<b\leq1$,
\[
\lim_{n\rightarrow+\infty} \frac{|\{1\leq m\leq n \text{ such that } \{x_{m}
\} \in(a,b)\}|}{n}=b-a,
\]
where the notation $\{x_{n}\}$ stands for the fractional part of $x_{n}$.

\begin{lemma}
\label{MMM}
One has
\[
{\mathcal{M}} = \left(  \frac{\pi}{6}, \frac{5\pi}{6} \right)  .
\]
\end{lemma}

\begin{proof}
Let $\omega$ be in $(\pi/6, 5\pi/6)$ (this implies $\sin\omega>1/2$). We
consider several distinct cases. \medskip

\noindent Case a) Suppose first that $\omega/ \pi$ is an irrational number.
Then the sequence $n \omega/ 2 \pi$ is uniformly distributed modulo $1$ (see
for instance Example 2.1 of \cite{KN}). Thus $n \omega$ is uniformly
distributed modulo $2\pi$.

Considering the quantity
\[
q_{\omega}=\frac{1}{2}\left(  1+\frac{1}{2\sin\omega}\right)  \in\left(
\frac{1}{2\sin\omega},1\right)
\]
the uniform distribution modulo $1$ property gives a positive density for
those $n$ such that
\[
n\omega\left(  \operatorname{mod} \pi\right)  \in(\arcsin q_{\omega}%
,\pi-\arcsin q_{\omega}),
\]
a given fixed interval. It follows that the set of $n$ such that
\[
\sin n\omega>q_{\omega}>\frac{1}{2}%
\]
has a positive lower density which concludes this case. \medskip

From now on, we consider the case where
\[
\omega= \frac{p \pi}{q}%
\]
for $p$ and $q$ two positive coprime numbers. Two cases remain to be
investigated. \medskip

\noindent Case b) $q$ even : recall that $p$ and $q$ are coprime integers. But
$q$ even implies $p$ odd and it is sufficient to notice that $p$ must be
invertible modulo $2q$. Call $p^{-1}$ any positive integer, inverse of $p$
modulo $2q$. Then the integers $n$ in
\[
\frac{qp^{-1}}{2} + \left(  2q\cdot{\mathbb{N}} \right)
\]
(notice that this set is an arithmetic progression of reason $2q$ and thus
(positive) density $1/2q$), satisfy $pn = q/2 + 2lq$ for some integer $l$ and thus
\begin{align*}
\sin n \omega &  = \sin\frac{np \pi}{q}\\
&  = \sin\left(  \frac12 +2l \right)  \pi\\
&  = \sin\left(  \frac{\pi}{2}+2l \pi\right) \\
&  = 1
\end{align*}
and this case is solved. \medskip

\noindent Case c) $q$ odd: we define
\[
\eta_{q} = \left\{
\begin{array}
[c]{rl}%
-1 & \text{ if } q \equiv1 \pmod{4}\\
1 & \text{ if } q \equiv3 \pmod{4}
\end{array}
\right.
\]
and notice that $q + \eta_{q}$ is divisible by $4$. Since $p$ is invertible
modulo $q$ (call $p^{-1}$ any positive integer, inverse of $p$ modulo $q$).
Then for any
\[
n\in\frac{p^{-1}(q+ \eta_{q})}{2}+ \left(  2q \cdot{\mathbb{N}} \right)
\]
(again a set with positive density, namely $1/2q$), one has $np=(q + \eta
_{q})/2 +2lq$ for some integer $l$. Thus we obtain
\[
\sin n\omega=\sin\left(  \frac{q + \eta_{q}}{2q}+2l \right)  \pi= \sin\left(
\frac{\pi}{2} + \frac{\eta_{q}}{2q} \pi\right)  = \cos\frac{\eta_{q} \pi}{2q}
=\cos\frac{ \pi}{2q}.
\]
Thus the result is proved if we can prove that
\begin{equation}
\sin\frac{p\pi}{q}\cos\frac{\pi}{2q}>\frac{1}{2}. \label{labelle}%
\end{equation}
For proving \eqref{labelle}, we notice that we can replace without loss of
generality $p$ by $q-p$. Therefore in what follows we assume
\[
\frac{p\pi}{q} \leq\frac{\pi}2.
\]

We define
\[
\rho= \frac{1}{2 \cos\pi/5}%
\]

First, if $\sin p \pi/q > \rho$, then we obtain
\[
\sin\left(  \frac{p\pi}{q} \right)  \cos\frac{\pi}{2q} > \frac{\cos\pi/2q}{2
\cos\pi/5} \geq\frac12
\]
since $q \geq3$ and \eqref{labelle} holds. \smallskip

Second, if
\[
\sin\frac{p\pi}{q} \leq\rho
\]
then applying the Mean-value Theorem yields
\begin{align*}
\sin\left(  \frac{p\pi}{q}\right)  -\frac{1}{2}  &  =\sin\frac{p\pi}{q}%
-\sin\left(  \frac{\pi}{6}\right) \\
&  \geq\left(  \frac{p\pi}{q}-\frac{\pi}{6}\right)  \min_{t\in\lbrack
\pi/6,p\pi/q]}\cos t\\
&  = \left(  \frac{p}{q}-\frac{1}{6}\right)  \pi\cos\frac{p\pi}{q}%
\end{align*}
since $p\pi/q \in( \pi/6,\pi/2 )$. We then obtain
\begin{align*}
\sin\left(  \frac{p\pi}{q}\right)  -\frac{1}{2}  &  \geq\frac{6p-q}{6q}%
\pi\sqrt{1-\rho^{2}}\\
&  \geq\frac{\pi}{6q}\sqrt{1-\rho^{2}},
\end{align*}
the last inequality following from the fact that the integer $6p-q$ 
is non-zero and thus $\geq1$. Moreover, using the classical
inequality $\cos x\geq1-x^{2}/2$ valid for $x$ positive, we infer that
\[
\sin\frac{p\pi}{q}\cos\frac{\pi}{2q}\geq\left(  \frac{1}{2}+\frac{\pi}%
{6q}\sqrt{1-\rho^{2}}\right)  \left(  1-\frac{\pi^{2}}{8q^{2}}\right)  .
\]
But the right-hand side is larger than or equal to $1/2$ as soon as
\[
1>\frac{3\pi}{8q\sqrt{1-\rho^{2}}}+\frac{\pi^{2}}{8q^{2}},
\]
an inequality valid for any $q\geq3$ which concludes of \eqref{labelle}.

The lemma is proved.
\end{proof}

Here is now what we can obtain.

\begin{lemma}
\label{lem3} Let $\vartheta$ be a non-zero (modulo $\pi/2$) real number of
double belonging to ${\mathcal{M}}$. Let $c_{\vartheta}>1/2$ be such that
${\mathcal{E}}_{2\vartheta, c_{\vartheta}}$ has a positive lower density.
Then, there is a positive real number $C^{\prime}\left(  \vartheta\right)  $
such that
\[
\left\vert t_{n}\left(  \vartheta\right)  \right\vert \geq C^{\prime}\left(
\vartheta\right)  \frac{4^{n}}{\sqrt{n}}%
\]
holds for any large enough $n$ in ${\mathcal{E}}_{2\vartheta, c_{\vartheta}}
-1 $.
\end{lemma}

We notice that it is well possible that the range of $\vartheta$ to which this lemma is applicable 
could be even extended.

\begin{proof}
We again distinguish two different cases.\smallskip

\noindent Case a) : we assume first that $n$ is odd. By \eqref{gggggg}
\[
t_{n}\left(  \vartheta\right)  =t_{n}^{(1)}\left(  \vartheta\right)
+\exp\left(  4in\vartheta\right)  t_{n}^{(1)}\left(  -\vartheta\right)
\]
where
\[
t_{n}^{(1)}\left(  \vartheta\right)  =\sum\limits_{k=0}^{(n-1)/2}\binom{2k}%
{k}\binom{2n-2k}{n-k}\exp\left(  4ik\vartheta\right)  .
\]
We now study the precise behavior of $t_{n}^{(1)}\left(  \vartheta\right)  $.
Recall
\[
t_{n}^{(1)}\left(  \vartheta\right)  =\sum\limits_{k=0}^{\left(  n-1\right)
/2}v_{k}\exp\left(  4ik\vartheta\right)
\]
where
\[
v_{k}=\binom{2k}{k}\binom{2n-2k}{n-k}.
\]
Recall that the sequence $\left(  v_{k}\right)  _{k\geq0}$ is decreasing for
$k\leq\left(  n-1\right)  /2$, $u_{-1}\left(  \vartheta\right)  =0$ and
$u_{k}\left(  \vartheta\right)  ={\sum\limits_{j=0}^{k}}\exp\left(
4ij\vartheta\right)  $ for $k\geq0$. We have
\[
\left\vert u_{k}\left(  \vartheta\right)  \right\vert =\left\vert \frac
{\sin\left(  2\left(  k+1\right)  \vartheta\right)  }{\sin2\vartheta
}\right\vert \leq\frac{1}{\left\vert \sin2\vartheta\right\vert }=c_{0}\left(
\vartheta\right)  .
\]
Abel's transformation gives
\begin{equation}
t_{n}^{(1)}\left(  \vartheta\right)  ={\sum\limits_{k=0}^{(n-1)/2-1}}\left(
v_{k}-v_{k+1}\right)  u_{k}\left(  \vartheta\right)  +u_{(n-1)/2}\left(
\vartheta\right)  v_{(n-1)/2}. \label{tn1}%
\end{equation}
We now define the sequence of positive real numbers ($k<n/2$)
\[
w_{k}=v_{k}-v_{k+1}=2\binom{2k}{k}\binom{2(n-k-1)}{n-k-1}\left(  \frac{1}%
{k+1}-\frac{1}{n-k}\right)  .
\]
We compute
\begin{align*}
x_{k}  &  =w_{k+1}-w_{k}\\
&  =4\binom{2k}{k}\binom{2(n-k-1)}{n-k-1}\\
&  \hspace{1.8cm} \times\left(  \frac{2}{(k+1)(n-k-1)} -\frac{3}%
{(k+1)(k+2)}-\frac{3}{(n-k)(n-k-1)}\right) \\
&  <12\binom{2k}{k}\binom{2(n-k-1)}{n-k-1}\left(  \frac{1}{(k+1)(n-k-1)}%
-\frac{1}{(k+1)(k+2)}\right) \\
&  =\frac{12}{k+1}\binom{2k}{k}\binom{2(n-k-1)}{n-k-1}\left(  \frac{1}%
{n-k-1}-\frac{1}{k+2}\right) \\
&  \leq0
\end{align*}
for $k\leq(n-3)/2$. It follows that the sequence $(w_{k})$ is decreasing up to
$(n-3)/2$.

Restarting from \eqref{tn1}, we obtain
\begin{align*}
t_{n}^{(1)}(\vartheta)  &  ={\sum\limits_{k=0}^{(n-1)/2-1}}w_{k}%
u_{k}(\vartheta)+u_{(n-1)/2}(\vartheta)v_{(n-1)/2}\\
&  =\sum_{k=0}^{(n-3)/2}w_{k}\frac{\exp(4i(k+1)\vartheta)-1}{\exp
(4i\vartheta)-1}+u_{(n-1)/2}(\vartheta)v_{(n-1)/2}\\
&  =\frac{1}{\exp(4i\vartheta)-1}\sum_{k=0}^{(n-3)/2}w_{k}(\exp
(4i(k+1)\vartheta)-1)+u_{(n-1)/2}(\vartheta)v_{(n-1)/2}\\
&  =\alpha(\vartheta)\sum_{k=0}^{(n-3)/2} w_{k}\exp(4ik\vartheta
)+\beta(\vartheta)\sum_{k=0}^{(n-3)/2}w_{k}+u_{(n-1)/2}(\vartheta
)v_{(n-1)/2}\\
&  =\alpha(\vartheta)\sum_{k=0}^{(n-3)/2}w_{k}\exp(4ik\vartheta)+\beta
(\vartheta)(v_{0}-v_{(n-1)/2})+u_{(n-1)/2}(\vartheta)v_{(n-1)/2}%
\end{align*}
where we define
\[
\alpha(\vartheta)=\frac{\exp(4i\vartheta)}{\exp(4i\vartheta)-1}\quad\text{ and
}\quad\beta(\vartheta)=-\frac{1}{\exp(4i\vartheta)-1}.
\]
We now perform a second Abel's transformation in the first term of this
expression and get
\begin{align*}
t_{n}^{(1)}(\vartheta)  &  =\alpha(\vartheta)\sum_{k=0}^{(n-3)/2}w_{k}%
\exp(4ik\vartheta)+\beta(\vartheta)(v_{0}-v_{(n-1)/2})+u_{(n-1)/2}%
(\vartheta)v_{(n-1)/2}\\
&  =\alpha(\vartheta)\left(  {\sum\limits_{k=0}^{(n-3)/2-1}}\left(
w_{k}-w_{k+1}\right)  u_{k}\left(  \vartheta\right)  +u_{(n-3)/2}\left(
\vartheta\right)  w_{(n-3)/2}\right) \\
&  \hspace{4.5cm}+\beta(\vartheta)(v_{0}-v_{(n-1)/2})+u_{(n-1)/2}%
(\vartheta)v_{(n-1)/2}\\
&  =\alpha(\vartheta)\left(  {\sum\limits_{k=0}^{(n-5)/2}}\left\vert
x_{k}\right\vert u_{k}\left(  \vartheta\right)  +u_{(n-3)/2}\left(
\vartheta\right)  w_{(n-3)/2}\right) \\
&  \hspace{4.5cm}+\beta(\vartheta)(v_{0}-v_{(n-1)/2})+u_{(n-1)/2}%
(\vartheta)v_{(n-1)/2}.
\end{align*}
One now checks that
\[
w_{(n-3)/2},v_{(n-1)/2}\leq C\frac{4^{n}}{n}%
\]
for some positive constant $C,$ which in view of the upper bound for the
$u_{i}(\vartheta)$ implies
\[
t_{n}^{(1)}(\vartheta)=\alpha(\vartheta)\left(  {\sum\limits_{k=0} ^{(n-5)/2}%
}\left\vert x_{k}\right\vert u_{k}\left(  \vartheta\right)  \right)
+\beta(\vartheta)v_{0}+O\left(  \frac{4^{n}}{n}\right)  .
\]

We thus deduce that%
\begin{align}
t_{n}\left(  \vartheta\right)   &  =t_{n}^{(1)}\left(  \vartheta\right)
+\exp\left(  4in\vartheta\right)  t_{n}^{(1)}\left(  -\vartheta\right)
\nonumber\\
&  =\alpha(\vartheta)\left(  {\sum\limits_{k=0}^{(n-5)/2}}\left\vert
x_{k}\right\vert u_{k}\left(  \vartheta\right)  \right)  +\beta(\vartheta
)v_{0}\nonumber\\
&  +\exp\left(  4in\vartheta\right)  \left(  \alpha(-\vartheta)\left(
{\sum\limits_{k=0}^{(n-5)/2}}\left\vert x_{k}\right\vert u_{k}\left(
-\vartheta\right)  \right)  +\beta(-\vartheta)v_{0}\right)  +O\left(
\frac{4^{n}}{n}\right) \nonumber\\
&  =\sum_{k=0}^{(n-5)/2}\big(\alpha(\vartheta)u_{k}(\vartheta)+\exp
(4in\vartheta)\alpha(-\vartheta)u_{k}(-\vartheta)\big)\left\vert
x_{k}\right\vert \label{eq23}\\
&  +\left(  \beta(\vartheta)+\exp\left(  4in\vartheta\right)  \beta
(-\vartheta)\right)  v_{0}+O\left(  \frac{4^{n}}{n}\right)  .\nonumber
\end{align}

From the identities
\begin{align}
\alpha\left(  -\vartheta\right)   &  =-\exp\left(  -4i\vartheta\right)
\alpha\left(  \vartheta\right) \nonumber\\
\beta\left(  -\vartheta\right)   &  =-\exp\left(  4i\vartheta\right)
\beta\left(  \vartheta\right) \nonumber\\
u_{k}\left(  -\vartheta\right)   &  =\exp\left(  -4ik\vartheta\right)
u_{k}\left(  \vartheta\right) \nonumber
\end{align}
and \eqref{eq23} we get
\begin{align*}
t_{n}\left(  \vartheta\right)   &  = \alpha(\vartheta)\sum_{k=0}^{(n-5)/2}
\left(  1-\exp(4i\left(  n-k-1\right)  \vartheta)\right)  u_{k}(\vartheta
)\left\vert x_{k}\right\vert \\
&  \hspace{2.5cm}+\left(  1-\exp(4i\left(  n+1\right)  \vartheta)\right)
\beta\left(  \vartheta\right)  v_{0}+O\left(  \frac{4^{n}}{n}\right) \\
&  = u_{n}\left(  \vartheta\right)  v_{0}-\exp\left(  4i\vartheta\right)
\sum_{k=0}^{(n-5)/2}u_{n-k-2}\left(  \vartheta\right)  u_{k}(\vartheta
)\left\vert x_{k}\right\vert +O\left(  \frac{4^{n}}{n}\right)  .
\end{align*}
Since%
\begin{equation}
\label{coscos}\left\vert u_{n-k-2}\left(  \vartheta)\right)  u_{k}\left(
\vartheta\right)  \right\vert \leq\frac{1}{\sin^{2}\left(  2\vartheta\right)
},
\end{equation}
we get
\begin{align*}
\left\vert \exp\left(  4i\vartheta\right)  \sum_{k=0}^{(n-5)/2}u_{n-k-2}
\left(  \vartheta\right)  u_{k}(\vartheta)x_{k}\right\vert  &  \leq\frac
{1}{\sin^{2}\left(  2\vartheta\right)  }\sum_{k=0}^{(n-5)/2}\left\vert
x_{k}\right\vert \\
&  =\frac{1}{\sin^{2}\left(  2\vartheta\right)  }\sum_{k=0}^{(n-5)/2}\left(
w_{k}-w_{k+1}\right) \\
&  =\frac{1}{\sin^{2}\left(  2\vartheta\right)  }\left(  w_{0}-w_{\frac
{n-3}{2}}\right) \\
&  \leq\frac{1}{\sin^{2}\left(  2\vartheta\right)  }w_{0}\\
&  =\frac{1}{\sin^{2}\left(  2\vartheta\right)  }\left(  v_{0}-v_{1}\right) \\
&  \sim\frac{1}{2\sin^{2}\left(  2\vartheta\right)  }v_{0},
\end{align*}
since $v_{1}\sim v_{0} /2$.

To summarize, in
\[
t_{n}\left(  \vartheta\right)  = u_{n} \left(  \vartheta\right)  v_{0} -
\exp\left(  4i\vartheta\right)  \sum_{k=0}^{(n-5)/2}u_{n-k-2}\left(
\vartheta\right)  u_{k}(\vartheta)\left\vert x_{k}\right\vert + O\left(
\frac{4^{n}}{n}\right)
\]
the first term is
\[
\sim\frac{\sin(2(n+1)\vartheta)}{\sin(2 \vartheta)} v_{0},
\]
the second one is
\[
\lesssim\frac{1}{2\sin^{2}\left(  2\vartheta\right)  }v_{0}%
\]
and the third one
\[
O\left(  \frac{v_{0}}{\sqrt{n}}\right)  .
\]
Therefore, if $n+1 \in{\mathcal{E}}_{2 \vartheta,c_{\vartheta}}$ is large
enough, we deduce that the sum satisfies
\[
\left\vert t_{n}(\vartheta)\right\vert \gg_{\vartheta}\frac{4^{n}}{\sqrt{n}}.
\]
\medskip

\noindent Case b) : If $n$ is now even, then
\[
t_{n}\left(  \vartheta\right)  =t_{n}^{(1)}\left(  \vartheta\right)
+\exp\left(  4in\vartheta\right)  t_{n}^{(1)}\left(  -\vartheta\right)
+\binom{n}{\frac{n}{2}}^{2}\exp\left(  2in\vartheta\right)
\]
where
\[
t_{n}^{(1)}\left(  \vartheta\right)  =\sum\limits_{k=0}^{\left(  n/2\right)
-1}\binom{2k}{k}\binom{2n-2k}{n-k}\exp\left(  4ik\vartheta\right)  .
\]
By similar computations as above, and using
\[
\binom{n}{\frac{n}{2}}^{2}\exp\left(  2in\vartheta\right)  =O\left(
\frac{4^{n}}{n}\right)  ,
\]
we get
\[
t_{n}\left(  \vartheta\right)  =u_{n}\left(  \vartheta\right)  v_{0}
-\exp\left(  4i\vartheta\right)  \sum_{k=0}^{(n-4)/2}u_{n-k-2}\left(
\vartheta\right)  u_{k}(\vartheta)\left\vert x_{k}\right\vert +O\left(
\frac{4^{n}}{n}\right)  .
\]
Similar arguments as above give the conclusion.
\medskip

Combining Case $a)$ and Case $b)$, we deduce the result announced.
\end{proof}

We can now conclude this section and prove our last result.

\begin{proof}[Proof of Theorem \ref{Main} (ii)]
Let $\vartheta \in (\pi/12, 5 \pi /12)$. 

On the one hand, by Proposition \ref{prop12}, we see that the element
$$
a_{\vartheta}=\left(
\begin{array}
[c]{cc}%
e^{i\vartheta} & 0\\
0 & e^{-i\vartheta}%
\end{array}
\right)
\notin N_{U}\left(  K \right).
$$

On the other hand, we have $2 \vartheta \in {\mathcal{M}}$ by Lemma \ref{MMM}. Thus there exists 
$c_\vartheta>1/2$ such that ${\mathcal{E}}_{2\vartheta, c_{\vartheta}}$ has a positive lower density.
Now, the series associated to $a_\vartheta$ by \eqref{eq2} is, in view of \eqref{series1} and \eqref{eq02},
\begin{eqnarray*}
\left\Vert \widehat{\mu}_{a_{\vartheta}}^{\left(  2\right)  }\right\Vert _{2}^2 & = &
 \sum_{n=1}^{\infty} (2n+1)\left\vert \frac{t_n ( \vartheta )  }{4^{n}}\right\vert ^{4}			\\
						& \geq & \sum_{n \in {\mathcal{E}}_{2\vartheta, c_{\vartheta}}-1} (2n+1)\left\vert \frac{t_n ( \vartheta )  }{4^{n}}\right\vert ^{4}\\	
						& \geq & C^{\prime}\left(  \vartheta\right)^4 \sum_{n \in {\mathcal{E}}_{2\vartheta, c_{\vartheta}}-1} \frac{2n+1}{n^2}
\end{eqnarray*}
by Lemma \ref{lem3}. It follows that
$$
\left\Vert \widehat{\mu}_{a_{\vartheta}}^{\left(  2\right)  }\right\Vert _{2}^2 \geq 2 C^{\prime}\left(  \vartheta\right)^4 \sum_{n \in {\mathcal{E}}_{2\vartheta, c_{\vartheta}}-1} \frac{1}{n},
$$
a series which is divergent by Lemma \ref{harmo} applied to the set ${\mathcal{E}}_{2\vartheta, c_{\vartheta}}-1$ which has a positive lower density.
\end{proof}

\end{document}